\documentclass[12pt]{amsart}
\usepackage[utf8]{inputenc}
\usepackage{amssymb}
\usepackage{graphicx}
\usepackage{graphicx,color}
\usepackage{latexsym}
\usepackage[all]{xy}
\usepackage{mathrsfs}
\usepackage{enumerate}
\usepackage{amssymb}
\usepackage{tikz}
\usepackage{breqn}
\usepackage{bm}
\usepackage{mathptmx}
\usetikzlibrary{shapes.geometric}

\makeatletter
\renewcommand{\subsection}{\@startsection
{subsection}{2}{0mm}{\baselineskip}{-0.25cm}
{\normalfont\normalsize\em}}

\makeatother

 \newtheorem{thm}{Theorem}[section]
\newtheorem{dfn}{Definition}[section]
\newtheorem{cor}[thm]{Corollary}
\newtheorem{prop}[thm]{Proposition}

\newtheorem{lemma}[thm]{Lemma}
\newtheorem{rem}[thm]{Remark}
\newtheorem{exam}[thm]{Example}

\def \Fb{{\bf F}}
\def \Fbc{\overline{\bf F}}

\numberwithin{equation}{section}

\title[Quantum codes from algebraic geometry codes]{Quantum codes from a new construction of self-orthogonal algebraic geometry codes}

\author{F.~Hernando}
\address{Universitat Jaume I (UJI),  Campus de Riu Sec, Institut Universitari de Matem\`atiques i Aplicacions de Castell\'o, 12071 Castell\' on
de la Plana, Spain.}
\email{carrillf@uji.es}

\author{G.~McGuire}
\address{UCD School of Mathematics and Statistics, University College Dublin, Dublin 4 (Ireland).}
\email{gary.mcguire@ucd.ie}

\author{F.~Monserrat}
\address{Instituto Universitario de Matemática Pura y Aplicada, Universidad Politécnica
de Valencia, Camino de Vera s/n, 46022 Valencia (Spain).}
\email{framonde@mat.upv.es}

\author{J.~J.~Moyano-Fern\' andez}
\address{Universitat Jaume I (UJI),  Campus de Riu Sec, Institut Universitari de Matem\`atiques i Aplicacions de Castell\' o , 12071 Castell\' on
de la Plana, Spain.}
\email{moyano@uji.es}

\begin{document}

\begin{abstract}
We present new quantum codes with good parameters which are constructed from self-orthogonal algebraic geometry codes. 
Our method permits a wide class of curves to be used in the formation of these codes.
These results demonstrate that there is a lot more scope for constructing self-orthogonal AG codes than was previously known.
\end{abstract}


\thanks{The second author was partially supported by Science
Foundation Ireland  Grant 13/IA/1914. The remainder authors were partially supported by the Spanish Government and the EU funding program FEDER, grants MTM2015-65764-C3-2-P and PGC2018-096446-B-C22. The first and fourth authors are also partially supported by Universitat Jaume I, grant UJI-B2018-10.}

\maketitle

\section{Introduction}

Polynomial time algorithms on quantum computers for integer prime factorization and discrete logarithms were given by Shor \cite{22RBC}. This justifies the great importance of quantum computation and, specifically, the relevance  of quantum error-correcting codes because they protect quantum information from decoherence and quantum noise. Over the last twenty-five years, error-correction has proved to be one of the main obstacles to scaling up quantum computing and quantum information processing. 
One of the first examples of a quantum error-correcting code is Shor's 9-qubit code \cite{shor}
which has been generalized in a series of many papers, including \cite{7kkk,8kkk,18kkk,calderbank98,Calderbank,38kkk,45kkk,AK,BE,35kkk,opt,71kkk}. Nowadays the theory of quantum error-correcting codes is a very active area of research (see \cite{lag3, lag2, QINP2,GHM-FFA,QINP,refer, GHR-IEEE} for some recent publications).

A classical linear error-correcting code is called \emph{self-orthogonal} if it is contained in its dual code. The CSS
(Calderbank-Shor-Steane) construction \cite{Calderbank, Steane}  showed that classical self-orthogonal codes with certain properties are useful in
the construction of quantum error-correcting codes.
As a result, people looking for good
quantum error-correcting codes started trying to 
 find classical self-orthogonal codes with the required properties.

In the 1970s and early 1980s, 
using concepts and tools coming from algebraic geometry, Goppa constructed error correcting linear codes from smooth and geometrically irreducible projective curves defined over a finite field (see \cite{AG, AG2, AFF-HS, Hoholdt}). They are called Goppa  or algebraic geometry (AG) codes and have played an important role in the theory of error-correcting codes. They were used to improve the Gilbert-Varshamov bound \cite{t-v-z} which was a
surprising result at that time.  In fact, every linear code can be realized as an algebraic geometry code \cite{Pellikaan}.
In the area of quantum information processing, what is important is that AG codes provide a natural  
context and method for finding classical self-orthogonal codes.
Thus,  researchers have focussed on finding suitable self-orthogonal AG codes 
because they 
 give rise to good quantum error-correcting codes.


Many of the properties of AG codes that give rise to good quantum error-correcting codes
were captured in the definition of Castle curves by Munuera, Sep\'ulveda, and Torres  \cite{Munuera2}.
In \cite{Quantum}, Munuera, Ten\'orio and Torres use the specific properties of algebraic geometry codes coming from Castle and weak Castle curves to provide new sequences of self-orthogonal codes. Essentially, they use Lemma 2 and Proposition 2 of \cite{Quantum} 
to provide those sequences.

 The main purpose of this paper is to show that there is a much larger family
of curves from which to obtain self-orthogonal AG codes and good quantum codes.
This family includes Castle curves.
As a demonstration we provide some examples and some families of curves giving sequences of one-point self-orthogonal AG codes which are not covered in \cite{Quantum}.

This paper is laid out as follows.
In Section \ref{sect2} we briefly summarize the construction of AG codes and establish some notation that will be used in the paper. In Section \ref{sect3} we state and prove the main theoretical results (Theorem \ref{theorem2} and  corollaries) that generalize the construction of Castle curves and will allow us to present the afore-mentioned sequences of self-orthogonal codes. The next sections are  devoted to applying these results and obtaining explicit families of curves giving rise to those sequences. In Section \ref{sect5}, we use them to obtain quantum codes with good parameters,
and we compare our results to previous papers.

In the numerical examples we use the computational algebra system MAGMA \cite{magma}.

\section{AG codes and their duals}\label{sect2}

Throughout this and next section, we fix an arbitrary finite field $\Fb$. Let $\chi$ be a nonsingular, projective and geometrically irreducible curve $\chi$ of genus $g$ over $\Fb$ (we will say simply `curve' for abbreviation). We write $\overline{\Fb}$ for an algebraic closure of $\Fb$ and $\chi(\Fb')$ denotes the set of $\Fb'$-valued points of $\chi$ for any field extension $\Fb'/\Fb$.

A \emph{divisor} of $\chi$ is a formal sum $D=\sum_{i=1}^r n_i P_i$, where $r$ is a positive integer, $P_i\in \chi(\overline{\Fb})$ and $n_i\in \mathbb{Z}\setminus \{0\}$ for all $i=1,\ldots,r$, and moreover $P_i\neq P_j$ if $i\neq j$. We will say that the divisor $D$ is \emph{$\Fb$-rational} if $D^\sigma=D$, where $D^{\sigma}:=\sum_{i=1}^r n_i\sigma(P_i)$, and $\sigma: \overline{\Fb}\rightarrow \overline{\Fb}$ is the Frobenius ${\Fb}$-automorphism.
 Equivalently, $D$ can be regarded as a linear combination of places of $\Fb(\chi)/ \Fb$  with integer coefficients \cite[Def. 1.1.8]{AFF-HS}, where $\Fb(\chi)$ denotes the function field of $\chi$. The \emph{support} of  $D$, denoted by $\mathrm{Supp}(D)$, is the set of points $\{P_1,\ldots, P_r\}$, and the \emph{degree} of $D$ is defined as $\deg(D):=\sum_{i=1}^r n_i \deg(P_i)$, where $\deg(P_i)$ denotes the cardinality of the orbit of $P_i$ under the action of $\sigma$ (or, equivalently, the degree of the extension $k(P_i)/ \Fb$, where $k(P_i)$ is the residue field of $P_i$). Notice that a point $P$ is \emph{$\Fb$-rational} (i.e.~$P\in \chi(\Fb)$) if and only if $\deg(P)=1$. 
 
For every rational function $f$ on $\chi$, not identically $0$, the \emph{divisor of} $f$ is
$$(f):=\sum_{P\in \chi(\overline{\Fb})} \nu_P(f)P$$
where, for each point $P\in \chi(\overline{\Fb})$, $\nu_P$ denotes the discrete valuation at $P$
defined as follows: for any $z$ belonging to the local ring ${\mathcal O}_{\chi,P}$ of $\chi$ at $P$, $\nu_P(z)$ is defined as the non-negative integer $m$ such that $z=ut^m$, $u$ being a unit and $t$ a generator of the maximal ideal of ${\mathcal O}_{\chi,P}$. A point $P\in \chi(\overline{\Fb})$ is said to be a \emph{zero} (resp. a \emph{pole}) of $f$ if $\nu_P(f)>0$ (resp., $\nu_P(f)<0$). Notice that $(f)=(f)_0-(f)_{\infty}$, where 
$(f)_0=\sum_{\nu_P(f)>0} \nu_P(f)P$ is the \emph{divisor of zeroes} of $f$ and
$(f)_{\infty}=\sum_{\nu_P(f)<0} \nu_P(f)P$) is the  \emph{divisor of poles}  of $f$.

A divisor $D$ as above is \emph{effective} if $n_i>0$ for all $i=1,\ldots,r$; we write then $D\geq 0$. Also, given two divisors $D$ and $D'$, the notation $D\geq D'$ means that the divisor $D-D'$ is effective. We also consider the following finite-dimensional $\Fb$-vector space associated with $D$:
\begin{align*}
{\mathcal L}(D):=&\{f\in \Fb(\chi)\mid D+(f)\geq 0\}\cup \{0\},
\end{align*}
where $(f)$ denotes the divisor associated to $f$. 
\medskip

For a fixed set of $\Fb$-rational points ${\mathcal P}:=\{P_1,P_2,\ldots,P_N\}$ on $\chi$, set $D:=P_1+P_2+\cdots+P_N$, and let $G$ be another $\Fb$-rational divisor of $\chi$ whose support is disjoint from ${\mathcal P}$. Consider the $\Fb$-vector space
$$
\Omega(D):=\{\omega\in \Omega(\chi)\mid (\omega)\geq D\}\cup \{0\},
$$
where $\Omega(\chi)$ is the $\Fb(\chi)$-vector space of rational differential forms over $\chi$, and $(\omega)$ denotes the divisor associated to any $\omega\in  \Omega(\chi)$.

\begin{dfn}
The \emph{AG code} associated to the triple $(\chi, D, G)$ is the linear code $C(D,G)$ of length $N$ over $\Fb$ given by the image of the linear map 
$$
ev_{\mathcal P}:{\mathcal L}(G)\rightarrow \Fb^N
$$ 
defined by $ev_{\mathcal P}(f):=(f(P_1), f(P_2),\ldots, f(P_N))$. 
\end{dfn}
It can be seen that its dual code, $C(D,G)^{\perp}$, coincides with the image of the map ${\rm res}_{\mathcal P}: \Omega(G-D)\rightarrow \Fb^N$ defined by ${\rm res}_{\mathcal P}(\omega)=({\rm res}_{P_1}(\omega), \ldots, {\rm res}_{P_N}(\omega))$, where ${\rm res}_{P_i}(\omega)$ stands for the residue of $\omega$ at $P_i$ for all $i=1,\ldots,N$. Furthermore, if $\omega$ is a differential form in $\Omega(\chi)$ with simple poles at $P_i$ and such that ${\rm res}_{P_i}(\omega)=1$ for all $i=1,\ldots,N$, then it holds that
$$
C(D,G)^{\perp}=C(D,(\omega)+D-G)
$$
(see, for instance, \cite[Lemma 1.38]{Duursma}).
Notice that a differential $\omega$ with these conditions does always exist.

\begin{dfn}
The code $C(D,G)$ is said to be \emph{self-orthogonal} if $C(D,G)\subseteq C(D,G)^\perp$. 
\end{dfn}

There is a  particular class of curves among those satisfying the definition of AG codes.
These are called Castle and weak Castle (pointed) curves, see \cite{Quantum, Munuera2}. A pointed curve is a pair $(\chi, P)$, where $\chi$ is a curve and $P\in\chi(\Fb)$ is a rational point on $\chi$.

Castle and weak Castle curves are defined taking into consideration the following notion. Let $\chi$ be a curve and $P$ an $\Fb$-rational point on $\chi$, and consider the valuation $v_P$ (attached to the local ring) at $P$. The set 
$$
\Gamma (P):=\Big \{-v_P(f) : f \in \bigcup_{k=0}^{\infty} \mathcal{L}(k P)\Big \}
$$
is an additive semigroup of $\mathbb{Z}$ which is called the Weierstra\ss~semigroup at the rational point $P$ of $\chi$.  We say that a pointed curve $(\chi, P)$ is \emph{Castle} if
\begin{enumerate}
\item $\Gamma (P)$ is symmetric, i.e., $h\in \Gamma (P)$ if and only if $2g-1-h \notin \Gamma (Q)$ for all $h$.
\item If $s:=\min \{ h \in \Gamma (P) : h \neq 0\}$, then $\# \chi (\Fb) = qs+1$.
\end{enumerate}
If we substitute condition (2) by
\begin{enumerate}
\item[(2')] There exist a morphism $\varphi:\chi \to \overline{\Fb} \cup \{\infty\}$ with $(\varphi)_{\infty}=\ell P$ as well as elements $a_1,a_2,\ldots , a_r\in \Fb$ such that $\varphi^{-1}(a_i)\subseteq \chi(\Fb)$ and $\# \varphi^{-1}(a_i)=\ell$ for all $i=1,\ldots, r$,
\end{enumerate}
then the pointed curve $(\chi, P)$ is said to be \emph{weak Castle}. Notice that the terminology makes sense, since Castle curves are always weak Castle curves \cite{Munuera2}.

\section{Main results}\label{sect3}

We start this section with some definitions and conventions. An \emph{affine plane curve over $\Fb$} will be a curve $C$ defined by an equation $g(x,y)=0$, where $g(x,y)\in \Fb[x,y]$, $(x, y)$ being affine coordinates. Considering projective coordinates $(X:Y:Z)$ such that $x=X/Z$ and $y=Y/Z$, we will denote by $\chi_C$  the projectivization of $C$, and by $\pi_C: \tilde{\chi}_C\rightarrow \chi_C$ the associated normalization morphism; in this way $\tilde{\chi}_C$ is a nonsingular model of $\chi_C$.

For every $a\in \Fb$, $L_a$ (resp., $L_{\infty}$) will denote the \emph{affine} line over $\Fb$ defined by the equation $x=a$ (resp., the \emph{projective} line over $\Fb$, called \emph{line at infinity}, with equation $Z=0$).

\begin{dfn}\label{oneplace}
{\rm
An affine plane curve $C$ over $\Fb$  \emph{is said to have only one place at infinity} if it is geometrically irreducible, there exists an $\Fb$-rational point $Q_{\infty}$ such that ${\chi}_C(\Fbc)\cap L_{\infty}(\Fbc)=\{Q_{\infty}\}$, $\chi_C$ has only one branch at $Q_{\infty}$ and this branch is defined over $\Fb$. We impose the additional condition that $C$ is not a line.
}

\end{dfn}

Notice that, in the situation of Definition \ref{oneplace}, there exists a unique point $P_{\infty}\in \tilde{\chi}_{C}(\Fbc)$ such that $\pi_C(P_{\infty})=Q_{\infty}$ and, moreover, $P_{\infty}$ is $\Fb$-rational. Since $\tilde{\chi}_{C}\setminus \{P_{\infty}\}$ and $C$ are isomorphic, we will identify the points of both curves.

\medskip

If $C_1$ and $C_2$ are affine or projective plane curves (with respective equations $A=0$ and $B=0$) and $Q$ is any point, then we write $I_Q(C_1,C_2)$ (and also $I_Q(A,B)$) for the intersection multiplicity of $C_1$ and $C_2$ at $Q$, see \cite[Def. 2.22]{Hoholdt}.  The intersection multiplicity is positive if and only if
$Q$ is a point on both $C_1$ and $C_2$.

\begin{dfn}
{\rm
Given two affine plane curves $C_1$ and $C_2$ over $\Fb$, we will say that $C_1$ and $C_2$ are \emph{transversal} if $I_Q(C_1,C_2)=1$ for all $Q\in C_1(\Fbc) \cap C_2(\Fbc)$. Also, we will say that $C_1$ and $C_2$ are \emph{$\Fb$-transversal} if they are transversal and, in addition, all the points in $C_1(\Fbc) \cap C_2(\Fbc)$ are $\Fb$-rational. }
\end{dfn}

Fixing a curve $C$, for every subset ${\mathcal A}$ of $\Fbc$, we will define  ${\mathcal P}_{\mathcal A}$ by
\[
{\mathcal P}_{\mathcal A} :=\{ (\alpha,\beta)\in C(\Fbc) : \alpha\in {\mathcal A} \}
\] 
and we will be studying the polynomial (where ${\mathcal A}$ is finite)
\[
f_{\mathcal A}(x):=\prod_{a\in {\mathcal A}} (x-a)
\]
and its derivative $f'_{\mathcal A}(x)$.
We will consider the divisor of zeros of the rational function $f'_{\mathcal A}(x)$, 
and if  
$$(f'_{\mathcal A} (x))_0=c_1Q_1+\cdots+c_sQ_s +mP_\infty$$ 
where the $Q_i$ are points in the affine chart and $P_\infty$ is the point at infinity of the curve, 
then we define a divisor $M$  by $M=c_1Q_1+\cdots
+c_sQ_s =(f'_{\mathcal A}(x))_0-mP_\infty$.
 It is easy to show that the divisor $M$ is $\Fb$-rational. We call $M$ the \emph{divisor of affine zeroes} of the rational function defined by the derivative $f'_{\mathcal A}(z)$.

\begin{thm}\label{theorem2}
Let $C$ be a smooth affine plane curve over $\Fb$ with only one place at infinity. Let $g$ be the genus of $\tilde{\chi}_C$ and
let $${\mathcal A}=
\{a\in \Fb\mid \mbox{$C$ and $L_a$ are $\Fb$-transversal}\}.
$$
Let $f_{\mathcal A}(z):=\prod_{a\in {\mathcal A}} (z-a)\in \Fb[z]$.
 Let $M$ be the divisor of affine zeroes of the rational function of $\tilde{\chi}_C$ defined by the derivative $f'_{\mathcal A}(z)$, as defined above.


Then the following hold:

\begin{itemize}
 \item[(a)] If $D$ is the divisor $\sum_{P\in {\mathcal P}_{\mathcal A}} P$, and $G$ is another $\Fb$-rational divisor such that $\mathrm{Supp}(G)\cap \mathrm{Supp}(D)=\emptyset$, then
$$
C(D,G)^{\perp}=C(D,(2g-2+\deg(D)-\deg(M))P_{\infty}+M-G).
$$
\item[(b)] If, in addition,  $2G\leq(2g-2+\deg(D)-\deg(M))P_{\infty}+M$ then 
 $C(D,G)\subseteq C(D,G)^{\perp}$.
\end{itemize}

\end{thm}

\begin{proof}

\begin{itemize}

\item[(a)] For all $Q=(a,b)\in {\mathcal P}_{\mathcal A}$ let $x_Q:=x-a$.  In view of the 
choice of ${\mathcal A}$, the image of $x_Q$ at the local ring at $Q$ is a uniformizing parameter. Consider the following  differential form of $\tilde{\chi}_C$:
\[
\omega = \biggl( \sum_{a\in {\mathcal A}} \frac{1}{x-a} \biggr) dx.
\]
Clearly, for any $P=(\alpha,\beta)\in {\mathcal P}_{\mathcal A}$, we have
$$
\omega=\biggl( \sum_{a\in {\mathcal A}} \frac{1}{x_P+\alpha-a} \biggr) dx_P=\frac{f'_{\mathcal A}(x_P+\alpha)}{\prod_{a\in {\mathcal A}} (x_P+\alpha-a)}dx_P.
$$
Therefore $\omega$ has poles at the points of ${\mathcal P}_{\mathcal A}$, which are of order 1 and have residue 1. Since $C$ and $L_a$ are $\Fb$-transversal for every root $a$ of $f_{\mathcal A}$, the associated divisor to $\omega$ is
$$
(\omega)=(\deg(D)+2g-2-\deg(M))P_{\infty} -D+M,
$$
and the result now follows from \cite[Th. 2.72]{Hoholdt}.
\item[(b)] It follows immediately from  (a).
\end{itemize}
\end{proof}

The following corollary (that is straightforward from Theorem \ref{theorem2}) concerns AG codes defined from divisors of type $G=mP_{\infty}$ and yields a range of values of $m$ for which the associated code is self-orthogonal.

\begin{cor}\label{coro}
Assume the notation and hypotheses of Theorem \ref{theorem2} and suppose that $G=mP_{\infty}$ with $m\in \mathbb{N}$. Then
$C(D,G)\subseteq C(D,G)^{\perp}$ if  $2m\leq 2g-2+\deg(D)-\deg(M)$.

\end{cor}

In the specific case of curves defined by a separable equation $F(y)=H(x)$, the degree of the divisor $M$ mentioned in the statement of Theorem \ref{theorem2} can be explicitly computed from the equation of $C$ and the degree of the polynomial $f'_{\mathcal A}$:

\begin{cor}\label{coro1}
Assume the notation and hypotheses of Corollary \ref{coro} and suppose that $C$ has an equation of the type $F(y)=H(x)$, where $F,H$ are polynomials with coefficients in $\Fb$. Then $\deg(M)=\deg(f'_{\mathcal A})\cdot \deg(F)$. 

Furthermore,
$C(D,G)\subseteq C(D,G)^{\perp}$ if  $$2m\leq 2g-2+\deg(D)-\deg(f'_{\mathcal A})\cdot \deg(F).$$
\end{cor}

\begin{proof}

Let $a_1,\ldots,a_r\in \Fbc$ be the distinct roots of the polynomial $f'_{\mathcal A}(z)$ and consider the decomposition $f'_{\mathcal A}(z)=\delta\sum_{i=1}^r (z-a_i)^{k_i}$, $\delta\in \Fbc\setminus \{0\}$. For each $i=1,\ldots,r$, let $b_{1,i},\ldots,b_{s_i,i}$ be the different roots of $F(y)-H(a_i)$ and consider the decomposition
$$
F(y)-H(a_i)=\prod_{j=1}^{s_i} (y-b_{j,i})^{\gamma_{j,i}}.
$$ 
Notice that the points in the support of $M$ are those in the set $\{Q_{i,j}:=(a_i,b_{j,i})\}_{1\leq i\leq r;\; 1\leq j\leq s_i}$.

The coefficient in $M$ of  one of the points $Q_{i,j}$ is $\nu_{Q_{i,j}}(f'_{\mathcal A}(x))$, where $\nu_{Q_{i,j}}$ is the valuation defined by the curve $C$ at $Q_{i,j}$; then
$$
\nu_{Q_{i,j}}(f'_{\mathcal A}(x))=k_i\cdot I_{Q_{i,j}}(F(y)-H(x),f'_{\mathcal A}(x))=k_i\cdot \gamma_{j,i}\cdot I_{Q_{i,j}}(y-b_{j,i},x-a_i)=k_i\cdot \gamma_{j,i},
$$
therefore
$$
\deg(M)=\sum_{i=1}^r k_i\sum_{j=1}^{s_i}  \gamma_{j,i}=\sum_{i=1}^r k_i\deg(F)=\deg(f'_{\mathcal A})\cdot \deg(F).
$$

The last part of the statement follows from Corollary \ref{coro}.
\end{proof}

\begin{rem}
{\rm
In practice, the main difficulty in applying Corollary \ref{coro1}
is that the polynomial $f_{\mathcal A}(z)$ and its derivative 
need to be known and can be hard to compute.
We give an example of this now. 

\begin{exam}\label{exam:max}
The curve $y^3-y=x^2-x^{10}$ 
has $1215$ affine rational points over $\mathbb{F}_{3^6}$.
The polynomial $f_{\mathcal A}(z)$ can be computed using MAGMA and has degree 405.
Furthermore, its derivative has degree 324.
Applying Corollary \ref{coro1} gives self-orthogonal curves for $m$ in the range
$17 \le m \le 129$.
\end{exam}

This is an interesting example because the curve is maximal (recall that a curve defined over $\mathbb{F}_q$
of genus $g$ is maximal over $\mathbb{F}_q$
if the number of projective $\mathbb{F}_q$-rational points  is equal to 
$q+1+2g\sqrt{q}$, see \cite{AFF-HS} ). 
Maximal curves are desirable in coding theory because the length of the corresponding codes
is very good.

We are unable to compute $f_{\mathcal A}(z)$ by hand in this example.
In the next sections we will give some infinite families of curves where we are able to compute $f_{\mathcal A}(z)$ by hand.
}
\end{rem}

Next we present a special case of  Corollary \ref{coro1}, where the range of values of $m$ for which the codes $C(D,mP_{\infty})$ are self-orthogonal depends only on the genus of $C$ and $\deg(F)$.
This bound can be used when   $f_{\mathcal A}(z)$ is not known.

\begin{cor}\label{coro2}
Assume the notation and hypotheses of Corollary \ref{coro} and suppose that $C$ has an equation of the type $F(y)=H(x)$, where $F,H$ are polynomials with coefficients in $\Fb$. Then
$C(D,G)\subseteq C(D,G)^{\perp}$ if  $$2m\leq 2g-2+\deg (F).$$ 
\end{cor}

\begin{proof}
First we will prove that $\deg(D)=\#{\mathcal A}\cdot \deg (F)$. Notice that $\deg(D)$ coincides with the cardinality of ${\mathcal P}_{\mathcal A}$; hence it is enough to show that $\#{\mathcal P}_{\{a\}}=\deg(F)$ for every $a\in {\mathcal A}$. For this purpose, notice that $I_P(C,L_a)=1$ for all $P\in {\mathcal P}_{\{a\}}$ because $C$ and $L_a$ are transversal. Then
\begin{align*}
\#{\mathcal P}_{\{a\}}&=\sum_{P\in {\mathcal P}_{\{a\}}} I_P(C,L_a)=\sum_{P\in {\mathcal P}_{\{a\}}} I_P(F(y)-H(x),x-a)\\
&=\sum_{P=(a,b)\in {\mathcal P}_{\{a\}}} I_P(F(y)-H(a),x-a)\\
&=\sum_{P=(a,b)\in {\mathcal P}_{\{a\}}} I_P(y-b,x-a)=\deg(F),
\end{align*}

where the last two equalities are deduced from the fact that $C$ and $L_a$ are $\Fb$-transversal.

Finally, the result follows from 
\begin{align*}
\deg(M)&= \deg(f'_{\mathcal A})\cdot \deg(F)\\
&\leq (\deg(f_{\mathcal A})-1)\cdot\deg(F)=(\#{\mathcal A}-1)\cdot\deg(F)=\deg(D)-\deg(F),
\end{align*}
where the first equality is consequence of Corollary \ref{coro1}.
\end{proof}

\begin{rem}
{\rm
There are examples where this bound is tight, in the sense that 
$C(D,G)\subseteq C(D,G)^{\perp}$ when  $2m\leq 2g-2+\deg (F)$,
and $C(D,G)\not\subseteq C(D,G)^{\perp}$ for the smallest $m$ with $2m> 2g-2+\deg (F).$
One example is 
$y^{27}-y=x^2$ over $\mathbb{F}_{3^6}$, 
the number of rational points is $N=1431+1$. 
The derivative $f'_{\mathcal A}(z) = 2z^{52} + 1$ so it is not constant.
The genus is $13$ and $\deg (F)=27$
so $2m \le 2g-2+\deg (F)$ becomes $m\le 25$.
We confirm with MAGMA that for  $1<m\le  25$  we have that $C(D,G)\subseteq C(D,G)^{\perp}$
but not for $m=26$. 
}
\end{rem}

To finish this section, we prove that the AG codes coming from Corollary \ref{coro1} arise from weak Castle curves.

\begin{prop}\label{wc}
If $C$ is a curve satisfying the hypotheses of Corollary \ref{coro1} then the pointed curve $(\tilde{\chi}_C,P_{\infty})$ is weak Castle.

\end{prop}

\begin{proof}
Assume the notation of Theorem \ref{theorem2} and suppose, without loss of generality, that $0\in {\mathcal A}$. 

Consider an arbitrary element $a\in {\mathcal A}$ and the divisor $(x-a)$ of the rational function $x-a$. Since $L_a$ and $C$ are $\Fb$-transversal one has that ${\mathcal P}_{\{a\}}\subseteq \tilde{\chi}_C(\Fb)$ and
$$(x-a)=\sum_{P\in {\mathcal P}_{\{a\}}}P-(\rho-\eta_a)P_{\infty},$$
where $\rho:=I_{Q_{\infty}}(L_{\infty},\chi_C)$ and, for every $a\in {\mathcal A}$, $\eta_a$ equals 
$I_{Q_{\infty}}(\chi_{L_a},\chi_C)$ 
if $Q_{\infty}$ belongs to $\chi_{L_a}$, and 0 otherwise. 

Notice that, independently of $a\in {\mathcal A}$, the point $Q_{\infty}$ belongs to $\chi_{L_a}$ if and only if $Q_{\infty}=(0:1:0)$; moreover, in this case, $I_{Q_{\infty}}(\chi_{L_a},\chi_C)$ equals ${\rm mult}_{Q_{\infty}}(\chi_C)$ (the multiplicity of $\chi_C$ at $Q_{\infty}$) because the line $L_a$ is not tangent to $\chi_C$ at $Q_{\infty}$ (notice that $C$ is not a line). This shows that the value $\eta_a$ does not depend on $a$ and that $\rho-\eta_a> 0$. Therefore
$$(x-a)_0=\sum_{P\in {\mathcal P}_{\{a\}}}P\;\;\mbox{ and }\;\; (x-a)_{\infty}=(\rho-\eta_0)P_{\infty}.$$
In particular, $\#{\mathcal P}_{\{a\}}=\rho-\eta_0$.

Now, consider the  morphism $f:\tilde{\chi}_C\rightarrow \mathbb{P}^1$ associated with the rational function defined by $x$. 
From the previous paragraphs, it holds that $(f)_{\infty}=(\rho-\eta_0)P_{\infty}$ and, for all 
$a\in {\mathcal A}$, $f^{-1}(a)={\mathcal P}_{\{a\}}\subseteq \tilde{\chi}_C(\Fb)$ and $\#f^{-1}(a)=\rho-\eta_0$. Hence, taking into account \cite[Prop. 3 (2)]{Quantum}, the pointed curve $(\tilde{\chi}_C,P_{\infty})$ is weak Castle.
\end{proof}

\begin{rem}\label{comment}
{\rm
We would like to comment on how our results differ from the results in \cite{Quantum} and \cite{G}. All the families of curves in \cite{Quantum} satisfy the hypotheses of Lemma 2 in that paper. 
Under the assumptions (and notation) of Theorem \ref{theorem2}, the pointed curve $(\tilde{\chi}_C,P_{\infty})$ satisfies the hypotheses of \cite[Lemma 2]{Quantum} if and only if the polynomial $f'_{\mathcal A}(z)$ is a nonzero constant (if and only if the divisor $M$ in Theorem \ref{theorem2} is  the zero divisor).
In this paper we will present some families with non-constant derivative, which are the
first of this kind as far as we are aware.

To emphasize this point, we partition the curves satisfying the hypotheses of 
Theorem \ref{theorem2} into two types: \\
Type I: those where $f'_{\mathcal A}(z)$ is a nonzero constant.\\
Type II: those where $f'_{\mathcal A}(z)$ is not constant.

The curves in  \cite{Quantum} are of Type I and many of the codes introduced in our paper
come from curves of Type II.  
Therefore, we are presenting a new type of code.
By Proposition \ref{wc} both types of  curves are weak Castle.
Most of the Type II curves in this paper are not Castle, as we will see.

The curves provided in \cite{G} are either of Type I or are not one-point AG codes.
All codes in our paper are one-point AG codes, and hence our results and examples
are different from \cite{G}.
Also, all the sets $ {\mathcal A}$  in \cite{G} are multiplicative subgroups after removing 0.
}
\end{rem}

\subsection*{Families of self-orthogonal AG codes}\label{sect4}

The aim of this subsection is to provide 
a lemma which will allow us to obtain several families of curves satisfying the hypotheses of Corollary \ref{coro1} and, therefore, to obtain families of self-orthogonal AG codes.

\begin{lemma}\label{infinity}
Let $\Fb$ be a finite field of characteristic $p$ and let $C$ be an affine plane curve over $\Fb$ with equation 
$$F(y)=H(x),$$
where $F$ and $H$ are polynomials with coefficients in $\Fb$ such that $F'(y)$ is a nonzero constant and $\gcd(\deg(H),p)=1$. Then 
\begin{itemize}
\item[(a)] $C$ is smooth.
\item[(b)] If $\deg(H)>\deg(F)$ or $H(x)=x^{\ell}$ with $\ell\in \mathbb{N}$ such that $\ell<\deg(F)$ and moreover $\gcd(\deg(F),\ell)=1$, then $C$ has only one place at infinity.
\item[(c)] The genus of $\tilde{\chi}_C$ is $\frac{1}{2}(\deg(F)-1)(\deg(H)-1)$.
\end{itemize}
\end{lemma}

\begin{proof}
Statement (a) is obvious, since the partial derivative with respect to $y$ of the defining equation of $C$ is a nonzero constant. We split the proof of (b) in two cases:\medskip

\emph{Case 1: $\deg(H)>\deg(F)$}. In this case, $(0:1:0)$ is the unique intersection point of $\chi_C$ and the line at infinity. Set $L:=F(y)-H(x)$, $F_0:=y$, $F_1:=x$, $\delta_0=d_1:=\deg(H)$, $\delta_1:=\deg_y\; Res_x(L,F_1)$ and $d_2:=\gcd(\delta_0,\delta_1)$, where $Res_x(L,F_1)$ denotes the resultant (with respect to $x$) of $L$ and $F_1$.

It is easily checked that $Res_x(L,F_1)=\pm(F(y)- H(0))$. Therefore, $\delta_1=\deg(F)$ and $d_2=1$.
Since $d_2=1$ and $\frac{d_1}{d_2}\delta_1$ is a multiple of $\delta_0$, Proposition 3.5 of \cite{Campillo-Farran} (see also the original source \cite{Abhyankar} by Abhyankar) implies that $C$ has only one place at infinity.

\medskip

\emph{Case 2: $H(x)=x^{\ell}$ with $\ell<\deg(F)$ and $\gcd(\deg(F),\ell)=1$}. In this case, $Q:=(1:0:0)$ is the unique intersection point of $\chi_C$ and the line at infinity. Setting $m:=\deg(F)$ one has that the equation of $\chi_C$ (in projective coordinates $X,Y$ and $Z$) is
$$
a_0Y^m+a_1Y^{m-1}Z+\cdots + a_1YZ^{m-1}+a_mZ^m-X^{\ell}Z^{m-\ell}=0,
$$
where $a_i\in \Fb$ for all $i=0,\ldots,m$ and $a_0\neq 0$. Taking coordinates $u:=Y/X$ and $v:=Z/X$ in the affine chart $U$ defined by $X\neq 0$ (to which $Q$ belongs), the equation of the restriction of $\chi_C$ to $U$ has the form
$$
h(u,v)-v^{m-\ell}=0,
$$
where $h$ is an homogeneous polynomial of degree $m$ such that $h(1,0)\neq 0$ and $Q$ is the origin. Hence, $C$ has a unique tangent at $Q$ (defined by $v=0$). Performing finitely many successive quadratic transformations we can obtain a resolution of singularities of $C$ at $Q$ (so that, by composition of them, we get the normalization morphism $\pi_C:\tilde{\chi}_C\rightarrow \chi_C$); see e.g.~\cite[Lecture 18]{Abhyankar2}. The quadratic transformation (with center $Q$) defined by $u=u'$ and $v=u'v'$ gives rise to the following equation of the proper transform $C'$ of $C$:
$$
(u')^{\ell}h(1,v')-(v')^{m-\ell}=0.
$$
Hence, $C'$ meets the exceptional line at a point that is $\Fb$-rational. Since $\gcd(\ell, m-\ell)=1$, it is not difficult to see that all the proper transforms involved in the process meet each exceptional line at a unique $\Fb$-rational point, and that the last proper transform has multiplicity one at every point. Since the points of $\tilde{\chi}_C$ are in one-to-one correspondence with the branches of $\chi_C$ \cite[Th. 5.29]{hkt}, it follows that $C$ has only one branch at $Q$ (which is $\Fb$-rational). 

It only remains to prove that $\chi_C$ is geometrically irreducible. Indeed, reasoning by contradiction, assume that $\chi_1$ and $\chi_2$ are two different components of $\chi_C$. Then both curves $\chi_1$ and $\chi_2$ must meet at the point $Q$, which contradicts the conclusion of the preceding paragraph.
\medskip

Statement (c) follows from \cite[Prop. 3]{Quantum}.
\end{proof}

Next, in Sections \ref{Family1}, \ref{Family2}, and \ref{Family3}, 
we will present some families of curves where our results are applicable. 
From now on, $q$ will be a power of a prime number $p$ and $N(C,q^n)$ stands for the number of $\mathbb{F}_{q^n}$-rational points of an affine curve $C$. We will make use of the notion of trace of an element $a\in \mathbb{F}_{q^n}$ over $\mathbb{F}_{q}$: the trace is the sum of the conjugates of $a$ with respect to $\mathbb{F}_q$, i.e.
$$
\mathrm{Tr}_{\mathbb{F}_{q^n}/\mathbb{F}_{q}}=a + a^q + \cdots + a^{q^{n-1}}.
$$

\section{Curves $A_{n,q,\ell}$}\label{Family1}

Let $\ell$ and $n$ denote positive integers (not both equal to 1) such that $\gcd(p,\ell)=1$, and let $A_{n,q,\ell}$ be the affine curve (defined over $\mathbb{F}_{q^n}$) with equation 
$$
y^{q^{n-1}}+y^{q^{n-2}}+\cdots +y=x^{\ell}.  
$$
The following Proposition refers to the statement of Theorem  \ref{theorem2}.


\begin{prop}\label{prop:8}
$(1)$ Let  $C=A_{n,q,\ell}$ and let $\Fb=\mathbb{F}_{q^n}$.
Then $C$ is smooth over $\Fb$  and $C$ has only one place at infinity.
The set  ${\mathcal A}$ in the statement 
of Theorem \ref{theorem2} is equal to the set of all
 $x$-coordinates of the $\mathbb{F}_{q^n}$-rational points of $C$. 
 
$(2)$  Moreover,
$f_{\mathcal A}(z)=z^{e+1}-z$, where $e:=\gcd(\ell (q-1), q^n-1)$, and the number of $\mathbb{F}_{q^n}$-rational points of $A_{n,q,\ell}$ is $N(A_{n,q,\ell},q^n)=q^{n-1}\cdot\left(e+1\right)$. 
\end{prop}

\begin{proof} 
$(1)$
By Lemma \ref{infinity}, $C$ is a smooth affine curve having one place at infinity with genus $g=(q^{n-1}-1)(\ell-1)/2$. 
If $a$ is the $x$-coordinate of an $\mathbb{F}_{q^n}$-rational point of the curve $A_{n,q,\ell}$ then the equation $\mathrm{Tr}_{\mathbb{F}_{q^n}/\mathbb{F}_{q}}(y)=a^{\ell}$ has $q^{n-1}$ distinct solutions for $y$ in $\mathbb{F}_{q^n}$. Hence all the points in the intersection $A_{n,q,\ell}(\overline{\mathbb{F}}_{q^n})\cap L_a(\overline{\mathbb{F}}_{q^n})$ are $\mathbb{F}_{q^n}$-rational. Moreover, if
 $Q=(a,b)$ is one of these points, then
$$
I_Q(A_{n,q,\ell},L_a)=I_{(0,b)}(y^{q^{n-1}}+y^{q^{n-2}}+\cdots+ y-a^{\ell},x)=1
$$
because $y-b$ is a simple factor of $y^{q^{n-1}}+y^{q^{n-2}}+\cdots+ y-a^{\ell}$. Therefore the set 
$$
\{a\in \mathbb{F}_{q^n}\mid \mbox{ there exists $b\in \mathbb{F}_{q^n}$ such that }(a,b)\in A_{n,q,\ell}(\mathbb{F}_{q^n})\}
$$ 
coincides with ${\mathcal A}=\{a\in \mathbb{F}_{q^n}\mid A_{n,q,\ell}\mbox{ and $L_a$ are $\mathbb{F}_{q^n}$-transversal}\}$.

$(2)$  Notice that, on the one hand, $0\in {\mathcal A}$. On the other hand, for every $a\in {\mathcal A}\setminus \{0\}$, we have $a^{\ell (q -1)} =1$ and $a^{q^n-1}=1$ and, therefore, $a$ is a root of $z^e-1$. Then every element of ${\mathcal A}$ is a root of $z^{e+1}-z$.

Conversely, let $a$ be a root of $z^e-1$. Then $a^{\ell(q-1)}=1$ and, therefore, $a^{\ell}\in \mathbb{F}_q$. Hence $a\in {\mathcal A}$ because the equation $\mathrm{Tr}_{\mathbb{F}_{q^n}/\mathbb{F}_{q}}(y)=a^{\ell}$ has solutions in $\mathbb{F}_{q^n}$ (by surjectivity of trace).

Finally, for every $x\in \mathbb{F}_{q^n}$, it holds that $x^{\ell}\in \mathbb{F}_q$ if and only if either $x=0$ or $x^{\ell(q-1)}=1$. Hence, since $y^{q^{n-1}}+y^{q^{n-2}}+\cdots +y$ is the image of $y$ by the trace of $\mathbb{F}_{q^n}$ over $\mathbb{F}_{q}$, we have
$$
N(A_{n,q,\ell},q^n)=q^{n-1}\cdot\left(\gcd(\ell (q-1), q^n-1)+1\right).
$$
\end{proof}

Proposition \ref{prop:8} means that we can apply Corollary  \ref{coro1} to the curve $A_{n,q,\ell}$,
and we deduce the following result:

\begin{cor}\label{ccc}
Let $N:=N(A_{n,q,\ell},q^n)$, let $\{P_1,\ldots,P_N\}$ be the set of $\mathbb{F}_{q^n}$-rational points of $A_{n,q,\ell}$, and 
let $D=P_1+\cdots+P_N$ be  a divisor of $\tilde{\chi}_{A_{n,q,\ell}}$. 
Then, for any nonnegative integer $m$, the AG code (defined from $\tilde{\chi}_{A_{n,q,\ell}}$) given by $C(D,m P_{\infty})$ is self-orthogonal if 
$$
2(m+1)\leq (q^{n-1}-1)(\ell-1)+q^{n-1}(\mu\cdot \gcd(\ell (q-1), q^n-1)+1),
$$
where $\mu:=1$ if $p$ divides $\gcd(\ell (q-1), q^n-1)+1$ and $\mu:=0$ otherwise.
\end{cor}

\begin{rem}
{\rm In \cite[Example~2]{Quantum} the authors consider curves $A_{n,q,\ell}$ with $\ell\mid (q^n-1)/(q-1)$ and show that, when $\ell \equiv 1 ({\rm mod}\; p)$, the pointed curves $(\tilde{\chi}_{A_{n,q,\ell}},P_{\infty})$ satisfy the hypotheses of Lemma 2 of \cite{Quantum}. Hence, in these cases, this lemma implies that the code $C(D, mP_{\infty})$ (defined as in Corollary \ref{ccc}) is self-orthogonal if 
$$
2(m+1)\leq (q^{n-1}-1)(\ell-1)+q^{n-1}(\ell (q-1)+1).
$$
Corollary \ref{ccc} gives a larger family of curves $A_{n,q,\ell}$ which do not necessarily satisfy the hypotheses of Lemma 2 of \cite{Quantum} (see Remark \ref{comment}). 

Lastly in this section, we show that the pointed curve $(\tilde{\chi}_{A_{n,q,\ell}},P_{\infty})$ is
almost never a Castle curve.
Proposition 2 of \cite{Quantum} can only be applied to $(\tilde{\chi}_{A_{n,q,\ell}},P_{\infty})$ 
when the curve is Castle.

Note that we never have $\ell = q^{n-1}$ because $\ell$ is relatively prime to $p$.

\begin{prop}
(1) If  $\ell < q^{n-1}$ the pointed curve $(\tilde{\chi}_{A_{n,q,\ell}},P_{\infty})$ is never a Castle curve.

(2) If  $\ell >q^{n-1}$ the pointed curve $(\tilde{\chi}_{A_{n,q,\ell}},P_{\infty})$ is  a Castle curve
if and only if \\
$\gcd(\ell, (q^n -1)/(q-1))=1$.
\end{prop}

\begin{proof}
Let $s$ be the smallest nonzero element of the Weierstra\ss~semigroup at $P_{\infty}$.
We know that the number of (affine) points is $q^{n-1}(e+1)$
so the curve is Castle if and only if $s=q^{n-2}(e+1)$.

Notice that $e$ is a multiple of $q-1$, since
$e=(q-1) \gcd(\ell, (q^n -1)/(q-1))$.

Proof of (1) : Suppose $\ell < q^{n-1}$.
In this case the smallest element of the Weierstra\ss~semigroup is $\ell$
i.e. $s=\ell$.
But we always choose $\ell$ to be relatively prime to $p$, so we cannot
have $\ell = q^{n-2}(e+1)$ for $n>2$.
Therefore the curve is never Castle in this case.

 If $n=2$ the curve is Castle iff  $\ell = e+1$.
Then $e=\ell -1$, but also $e=(q-1) \mathrm{gcd}(\ell,q+1)$.
If $\gcd(\ell,q+1)=1$ then $\ell = q$, which is  impossible.
If $\gcd(\ell,q+1)>1$ then there is a divisor of $\ell$ which is also a divisor of $\ell -1$, which is impossible.

Proof of (2) : Suppose $\ell > q^{n-1}$
In this case the smallest element of the Weierstra\ss~semigroup is $q^{n-1}$
i.e. $s=q^{n-1}$.
The curve is Castle if and only if $q=e+1$.
However  $e=q-1$ if and only if $\gcd(\ell, (q^n -1)/(q-1))=1$, by the definition of $e$.
\end{proof}


}

\end{rem}

\section{Curves $B_{q,G}$}\label{Family2}

 Let $n$ be a positive integer and consider a polynomial $G(x)\in \mathbb{F}_q[x]$ such that $\deg(G)>q$ and $\gcd(p,\deg(G))=1$. Consider the unique polynomial $\mathrm{Tr}_n(G)(z)\in \mathbb{F}_q[z]$ with degree at most $q^n-1$ such that $\mathrm{Tr}_n(G)(a)=\mathrm{Tr}_{\mathbb{F}_{q^n}/ \mathbb{F}_q}(G(a))$ for all $a\in \mathbb{F}_{q^n}$.  We will assume that 
 \begin{enumerate}
\item  $\mathrm{Tr}_n(G)$ is separable, and 
\item all  roots of $\mathrm{Tr}_n(G)$ belong to $\mathbb{F}_{q^n}$. 
 \end{enumerate}
 
 For such $G$, we define $B_{q,G}$ to be the affine curve (over $\mathbb{F}_q$) with equation $y^q-y=G(x)$. Notice that, by \cite[Th.~2.25]{lidl}, the set of $\mathbb{F}_{q^n}$-rational points of $B_{q,G}$ is $\{(a,b)\in \mathbb{F}_{q^n}^2 \mid  \mathrm{Tr}_n(G)(a)=0 \mbox{ and } b^q-b=G(a)\}$.

The following proposition refers to the statement of Theorem \ref{theorem2}.

\begin{prop}\label{prop:11}
$(1)$ Let  $C=B_{q,G}$ and let $\Fb=\mathbb{F}_{q^n}$.
Then $C$ is smooth over $\Fb$  and $C$ has only one place at infinity.
The set  ${\mathcal A}$ in the statement 
of Theorem \ref{theorem2} is equal to the set of all
 $x$-coordinates of the $\mathbb{F}_{q^n}$-rational points of $C$. 

$(2)$ Moreover
$$
f_{\mathcal A}(z)=\gamma\cdot  \mathrm{Tr}_n(G)(z)
$$
for some $\gamma\in \mathbb{F}_q\setminus \{0\}$, and the number of $\mathbb{F}_{q^n}$-rational points of $B_{q,G}$ is 
$
N(B_{q,G},q^n)=q\cdot \deg(\mathrm{Tr}_n(G)).
$
\end{prop}

\begin{proof}

$(1)$ First of all, notice that the curve $B_{q,G}$ is smooth and has only one place at infinity by Lemma \ref{infinity}.

Second, if $a$ is the $x$-coordinate of an $\mathbb{F}_{q^n}$-rational point of $B_{q,G}$, then the equation $y^q-y=G(a)$ has $q$ distinct solutions in $\mathbb{F}_{q^n}$. Indeed, since $y^q-y=G(a)$ has, at least, one solution $b\in \mathbb{F}_{q^n}$, it is obvious that the set of all solutions is $\{b+\alpha\mid \alpha\in \mathbb{F}_q\}$. Hence, an analogous reasoning as in the proof of Theorem \ref{theorem2} shows that $L_a$ and $B_{q,G}$ are $\mathbb{F}_{q^n}$ transversal. 

$(2)$ This follows from part (1) because $\mathrm{Tr}_n(G)$ is a separable polynomial and all its roots belong to $\mathbb{F}_{q^n}$. Finally, the counting of $\mathbb{F}_{q^n}$-rational points is easy to check.
\end{proof}

Using Proposition \ref{prop:11} we can apply Corollary  \ref{coro1} to the curve $B_{q,G}$ and deduce the following result:

\begin{cor}\label{ccc2}
Let $N:=N(B_{q,G},q^n)$, let $\{P_1,\ldots,P_N\}$ be the set of $\mathbb{F}_{q^n}$-rational points of $B_{q,G}$, and 
let $D=P_1+\cdots+P_N$ be  a divisor of $\tilde{\chi}_{B_{q,G}}$.

Then, for any nonnegative integer $m$, the AG code (defined from $\tilde{\chi}_{B_{q,G}}$) given by $C(D,m P_{\infty})$ is self-orthogonal if 
$$
2(m+1)\leq (q-1)(\deg(G)-1)+q\cdot (\deg(\mathrm{Tr}_n(G))- \deg(\mathrm{Tr}_n(G)')).
$$
\end{cor}

For the rest of this section we will consider the special case that 
$G(x)=H_k(x)$ where $H_k(x):=x^{q^k+1}+x$,
and $n=2k$, and $\gcd (n,p)=1$.
First we must verify the conditions on $H_k$ in order to apply Corollary \ref{ccc2}.

 \begin{lemma}
 Assume $q$ is odd, let $k$ be a positive integer such that $\gcd(p,2k)=1$, and let $n=2k$.  Then
(1)  $\mathrm{Tr}_n(H_k)$ is separable, and 
(2) all  roots of $\mathrm{Tr}_n(H_k)$ belong to $\mathbb{F}_{q^n}$. 
 \end{lemma}
 
 \begin{proof}
Notice that, for all $a\in \mathbb{F}_{q^n}$, it holds that
\begin{align*}
\mathrm{Tr}_{\mathbb{F}_{q^n}/ \mathbb{F}_q}(H_k(a))&=\mathrm{Tr}_{\mathbb{F}_{q^n}/ \mathbb{F}_q}(a^{q^k+1})+\mathrm{Tr}_{\mathbb{F}_{q^n}/ \mathbb{F}_q}(a)\\
&=2(a^{q^k+1} + a^{q^{k+1}+q}+a^{q^{k+2}+q^2}+\cdots +a^{q^{2k-1}+q^{k-1}})+\mathrm{Tr}_{\mathbb{F}_{q^n}/ \mathbb{F}_q}(a).
\end{align*}
Therefore: 
$$
\mathrm{Tr}_n(H_k)(z)=2(z^{q^k+1} + z^{q^{k+1}+q}+z^{q^{k+2}+q^2}+\cdots +z^{q^{2k-1}+q^{k-1}})+(z+z^q+z^{q^2}+\cdots +z^{q^{n-1}})
$$
and we can see that the degree of $\mathrm{Tr}_n(H_k)(z)$ is  $q^{n-1}+q^{k-1}$. 
Computing the derivative we have
$$
\mathrm{Tr}_n(H_k)'(z) = 2z^{q^k} +1= 2(z+1/2)^{q^k}.
$$
Notice that $\mathrm{Tr}_n(H_k)'$ has only one root (namely $-1/2$) which has multiplicity $q^k$. Moreover $H_k(-1/2)= -1/4$;
so $\mathrm{Tr}_n(H_k)(-1/2)=-n/4$, which is not zero because $p$ does not divide $n$. Hence $\mathrm{Tr}_n(H_k)$ is separable because it is relatively prime to its derivative.
This proves (1).

The number of $\mathbb{F}_{q^n}$-rational points of $B_{q,H_k}$ is
 $$N(B_{q,H_k},q^n)=q^{n}+q^k$$
  which is proved in \cite[Thm 20]{MY}.
 It then follows from the degree calculation above that  
  $$N(B_{q,H_k},q^n)=q\cdot \deg(\mathrm{Tr}_n(H_k)).$$
  Hence all the roots of $\mathrm{Tr}_n(H_k)$ belong to $\mathbb{F}_{q^n}$.
  This proves (2).
  \end{proof}
  
  Assume then that $q$ is odd, let $k$ be a positive integer such that $\gcd(p,2k)=1$ and consider the curve $B_{q,H_k}$ over the field $\mathbb{F}_{q^n}$ where $n:=2k$. 
The curve $B_{q,H_k}$ satisfies the hypotheses of Corollary \ref{ccc2}; and in addition
we have shown that 
$$
f_{\mathcal A}(z)=\frac{1}{2}\mathrm{Tr}(H_k)(z)\;\;\mbox{ and }\;\; f'_{\mathcal A}(z)=(z+1/2)^{q^k}.
$$

As a consequence we may apply Corollary \ref{ccc2} to the curves $B_{q,H_k} : y^q-y=x^{q^k+1}+x$.
We get that, for any positive integer $m$, 
the associated AG code $C(D,mP_{\infty})$ (with $D$ as in Corollary \ref{ccc2}) is self-orthogonal if 
\begin{equation}\label{diamonds}
2(m+1)\leq q^n.
\end{equation}

\begin{rem}
{\rm

Notice that none of the pointed curves $B_{q,H_k}$ discussed here satisfies Lemma 2 of \cite{Quantum} (see Remark \ref{comment}). None of the curves $B_{q,H_k}$  is Castle either, 
because the smallest element of the Weierstra\ss~semigroup is $q$, and the number of (affine) rational points is not
equal to $q^2$.
Hence \cite[Prop. 2]{Quantum} cannot be applied.

}
\end{rem}

\section{Curves $C_{q,\ell}$} \label{Family3}

Let $\ell$ be a positive integer such that $\gcd(p,\ell)=1$.
Let $C_{q^s,\ell}$ be the affine plane curve defined by the equation $y^{q^s}-y=x^{\ell}$. 
We consider the curve over $\mathbb{F}_{q^n}$.

We consider two special cases here, firstly when $s=1$ and $n=2$, and secondly
for arbitrary $s>1$ and $n$ with an extra hypothesis.

\subsection{Curves $C_{q, \ell}$}\label{Family31}
 Assume that $q$ is odd and $2\gcd(\ell,q+1)$ divides $q+1$. 
Let $C_{q,\ell}$ be the affine plane curve defined by the equation $y^q-y=x^{\ell}$. 
We consider the curve over $\mathbb{F}_{q^2}$.

The following Proposition refers to the statement of Theorem  \ref{theorem2}.

\begin{prop}\label{prop:50}
$(1)$ Let  $C=C_{q,\ell}$ and $\Fb=\mathbb{F}_{q^2}$.
Then $C$ is smooth over $\Fb$  and $C$ has only one place at infinity.
The set  ${\mathcal A}$ in the statement 
of Theorem \ref{theorem2} is equal to the set of all
 $x$-coordinates of the $\mathbb{F}_{q^2}$-rational points of $C$. 
 
$(2)$  Moreover,
$f_{\mathcal A}(z)=z^{e+1}-z$, where $e:=\gcd(\ell (q-1), q^2-1)$, and the number of 
$\mathbb{F}_{q^2}$-rational points of $C$ is $N(C_{q,\ell},q^2)=q\cdot\left(e+1\right)$. 
\end{prop}

\begin{proof}
By Lemma \ref{infinity} it holds  that $C_{q,\ell}$ is smooth, it has only one place at infinity, and the genus of $\tilde{\chi}_{C_{q,\ell}}$ is $(q-1)(\ell-1)/2$.
The set of $\mathbb{F}_{q^2}$-rational points of $C_{q,\ell}$ is
$$\{(a,b)\in \mathbb{F}_{q^2}\mid b^q-b=a^\ell\}$$
which implies $\mathrm{Tr}_{\mathbb{F}_{q^2}/\mathbb{F}_q}(a^\ell)=0$.
For each $a$ with $\mathrm{Tr}_{\mathbb{F}_{q^2}/\mathbb{F}_q}(a^\ell)=0$ 
there are $q$ solutions for $b$.
Then $a^\ell + a^{\ell q}=0$ which implies $a=0$ or $a^{\ell (q-1)}=-1$.
If $a\not=0$, 
the assumption $2\gcd(\ell,q+1)$ divides $q+1$ implies that
there are $e$ solutions for $a$, where $e:=\gcd(\ell(q-1),q^2-1)$. 
So $N(C_{q,\ell},q^2)=q(e+1)$.

Similar arguments to those given in Section \ref{Family1} for $A_{n,q,\ell}$ 
show that the curve $C_{q,\ell}$ satisfies the hypotheses of Theorem \ref{theorem2} for $\Fb=\mathbb{F}_{q^2}$, and that the set ${\mathcal A}$ consists of the $x$-coordinates of the $\mathbb{F}_{q^2}$-rational points.  Moreover, it is easy to check that
$$f_{\mathcal A}(z)=z^{e+1}-z.$$ 
\end{proof}

Proposition \ref{prop:50} means that we can apply Corollary  \ref{coro1} to the curve $C_{q,\ell}$,
and we deduce the following result:

\begin{cor}\label{ccc3}
Let $N:=N(C_{q,\ell},q^2)$, let $\{P_1,\ldots,P_N\}$ be the set of $\mathbb{F}_{q^2}$-rational points of $C_{q,\ell}$, and 
let $D=P_1+\cdots+P_N$ be  a divisor of $\tilde{\chi}_{C_{q,\ell}}$.

Then, for any nonnegative integer $m$, the AG code (defined from $\tilde{\chi}_{C_{q,\ell}}$) given by $C(D,m P_{\infty})$ is self-orthogonal if 
$$2(m+1)\leq (q-1)(\ell-1)+q(e+1)-\mu\cdot eq$$
where $\mu:=0$ if $p$ divides $e+1$ and $\mu:=1$ otherwise.
\end{cor}

\begin{rem}
{\rm
Notice that the derivative $f_{\mathcal A}'(z)$ is constant if and only if $p$ divides $e+1$.
Lemma 2 of \cite{Quantum} (see Remark \ref{comment}) can only  be applied if 
$p$  divides $e+1$.
Our result includes the case that $p$ does not divide $e+1$.

\begin{prop}
(1) If  $\ell < q$ the pointed curve $(\tilde{\chi}_{C_{q,\ell}},P_{\infty})$ is never a Castle curve.

(2) If  $\ell >q$ the pointed curve $(\tilde{\chi}_{C_{q,\ell}},P_{\infty})$ is  a Castle curve
if and only if \\
$\gcd(\ell, q+1)=1$.
\end{prop}

\begin{proof}
The curve is Castle if and only if $s=e+1$ where $s$ is the smallest
nonzero element of the Weierstra\ss~semigroup.
Also note that $e=(q-1)\gcd (\ell,q+1)$.

(1) If $\ell <q$ then $s=\ell$, so the curve is Castle if and only if $e=\ell -1$ which is impossible.

(2) If $\ell >q$ then $s=q$, so the curve is Castle if and only if $e=q-1$, which happens
if and only if $\gcd (\ell,q+1)=1$.
\end{proof}

Proposition 2 of \cite{Quantum} cannot be applied to $(\tilde{\chi}_{C_{q,\ell}},P_{\infty})$ if 
the curve is not Castle, however our result applies in all cases.
}
\end{rem}


\subsection{Curves $C_{q^s,\ell}$}\label{Family32} Let $s$ and $n$ be positive integers such that $n$ is a multiple of $s$ and $n/n_{\ell}^s$ is a multiple of $p$, where $n_\ell^s$ denotes the cardinality of the cyclotomic coset of
$\ell$ with respect to $q^s$, that is, the cardinality of the set $\{\ell{q^{js}} \mbox{mod }(q^n-1) \mid j=0,\ldots,n-1\}$.






The key fact in this case is that $\mathrm{Tr}_{q^n/q^s}(x^{\ell})=\frac{n}{n_{\ell}^s}(x^{\ell}+x^{\ell\cdot q^s}+\cdots+x^{\ell\cdot q^{(n_{\ell}^s-1)s}})$, and this is always 0 because $\frac{n}{n_{\ell}^s}\equiv 0 \pmod{p}$. Therefore any element of  $\mathbb{F}_{q^n}$ has trace equal to zero. So following same arguments as in previous subsections we have the following result.

\begin{prop}\label{prop:51}   Let  $C=C_{q^s,\ell}$ and $\Fb=\mathbb{F}_{q^n}$.
 If $\frac{n}{n_{\ell}^s}$ is divisible by $p$ then
 
$(1)$  $C$ is smooth over $\Fb$  and $C$ has only one place at infinity.
The set  ${\mathcal A}$ in the statement 
of Theorem \ref{theorem2} is equal to the set of all
 $x$-coordinates of the $\mathbb{F}_{q^n}$-rational points of $C$. 
 
$(2)$  Moreover, ${\mathcal A}=\mathbb{F}_{q^n}$ and
$f_{\mathcal A}(z)=z^{q^n}-z$, and the number of $\mathbb{F}_{q^n}$-rational points of $C$ is $N(C_{q^s,\ell},q^n)=q^{n+s}$.
\end{prop}

Hence, applying Corollary \ref{coro1} we have that the code over $\mathbb{F}_{q^n}$ given by $C(D,m P_{\infty})$ (with $D$ as in Corollary \ref{coro}) is self-orthogonal if
$$2(m+1)\leq (q-1)(\ell-1).$$

We note that these codes are of Type I, that is, the derivative of $f_{\mathcal A}(z)$ 
is constant.

 

%
%
%
%

\section{Application to quantum codes}\label{sect5}

In this section we will use the results of the previous sections to construct 
new quantum error-correcting codes. We point out that the number of rational points on our curves 
is always greater than the field size.
We will show that our curves beat the Gilbert-Varshamov bound.

Recall that the Hermitian inner product of any two vectors $\mathbf{x} =(x_1,x_2, \ldots, x_N)$ and $\mathbf{y}=(y_1,y_2, \ldots, y_N)$ in the vector space $\mathbb{F}_{q^2}^N$ is defined as $\mathbf{x}  \cdot_h \mathbf{y}= \sum x_i y_i^q$ and the Euclidean inner product of $\mathbf{x}$ and $\mathbf{y}$ in  $\mathbb{F}_{q}^N$ as $\mathbf{x} \cdot \mathbf{y} = \sum x_i y_i$.  Given a linear code $\mathcal{C}$ in $\mathbb{F}_{q^2}^N$ (respectively, $\mathbb{F}_{q}^N$), the Hermitian (respectively, Euclidean) dual space is denoted by $\mathcal{C}^{\perp_h}$ (respectively, $\mathcal{C}^\perp$).

In \cite{calderbank98} the following key theorem is stated and in  \cite{ketkar06} is generalized over any field.
\begin{thm}
\label{bueno}
The following two statements hold.
\begin{enumerate}
    \item Let $\mathcal{C}$ be a linear $[N,k,d]$ error-correcting  code over  $\mathbb{F}_{q}$ such that 
    $\mathcal{C} \subseteq \mathcal{C}^\perp$. Then, there exists an $[[N, N-2k, \geq d^\perp]]_q$ stabilizer quantum code, where $d^\perp$ denotes the minimum distance of $\mathcal{C}^\perp$. 
    If the minimum weight of $\mathcal{C}^{\perp}\setminus \mathcal{C}$ is equal to $d^\perp$, 
    then the stabilizer code is pure and has minimum distance $d^\perp$.
        \item Let $\mathcal{C}$ be a linear $[N,k,d]$ error-correcting  code over  $\mathbb{F}_{q^2}$ such that $\mathcal{C}\subseteq \mathcal{C}^{\perp_h }$. Then, there exists an $[[N, N-2k, \geq d^{\perp_h}]]_q$ stabilizer quantum code, where $d^{\perp_h}$ denotes the minimum distance of $\mathcal{C}^{\perp_h}$.
         If the minimum weight of $\mathcal{C}^{\perp_h}\setminus \mathcal{C}$ is equal to $d^{\perp_h}$, 
    then the stabilizer code is pure and has minimum distance $d^{\perp_h}$.
\end{enumerate}

\end{thm}

Recall that the stabilizer quantum code associated to $\mathcal{C}$, as in the previous theorem, is pure if the minimum distance of $\mathcal{C}^\perp$ (or $\mathcal{C}^{\perp_h}$) coincides with the minimum Hamming weight of $\mathcal{C}^\perp\setminus \mathcal{C}$
(or $\mathcal{C}^{\perp_h}\setminus \mathcal{C}$ ).

\begin{cor}\label{corolariobueno}
The following statements hold:
\begin{itemize}
\item[(1)] Let $\mathcal{C}$ be a linear $[N,k,d]$ error-correcting code over $\mathbb{F}_q$ such that $\mathcal{C}\subseteq \mathcal{C}^{\perp}$. If $d>k+1$ then there exists an $[[N,N-2k,\geq d^\perp]]_q$ stabilizer quantum code which is pure.
\item[(2)] Let $\mathcal{C}$ be a linear $[N,k,d]$ error-correcting code over $\mathbb{F}_{q^2}$ such that $\mathcal{C}\subseteq \mathcal{C}^{\perp_h}$. If $d>k+1$ then there exists an $[[N,N-2k,\geq d^{\perp_h}]]_q$ stabilizer quantum code which is pure.
\end{itemize}
\end{cor}

\begin{proof}
The result follows from Theorem \ref{bueno} and the fact $d^\perp \leq k+1$ (resp., $d^{\perp_h} \leq k+1$) by the Singleton bound.

\end{proof}

\subsection{Euclidean Inner Product}

Now we are going to consider codes within the framework of Theorem \ref{theorem2}, that is, codes $C(D,G)$ associated to curves with equation of the type $F(y)=H(x)$ such that $D=P_1+\cdots+P_N$ and $G=mP_{\infty}$, with $2g-2<m<N$ and $P_1,\ldots, P_N, P_{\infty}$ being rational points of the curve. The parameters of $C(D,G)$ are $[N,m-g+1,\ge N-m]$  (see \cite{Hoholdt}).


Moreover the dual code $C(D,G)^{\perp}=C(D,(2g-2+\deg(D)-\deg(M))P_{\infty}+M-G)$ has parameters
\begin{equation}\label{eq:dualparameters}
[N,N-m+g-1,\ge m-2g+2]
\end{equation}
Assuming self-orthogonality, Theorem \ref{bueno} provides a quantum code with parameters
\begin{equation}\label{eq:qparameters}
[[N,N-2(m-g+1),\ge m-2g+2]].
\end{equation}

Notice that, by Corollary \ref{corolariobueno}, this code is pure if
\begin{equation}\label{formulabuena}
N>2m-g+2.
\end{equation}

We notice here that all the forthcoming examples satisfy the above condition (\ref{formulabuena}) and, therefore, they are pure.

 With the same philosophy of the classical Gilbert-Varshamov bound, a sufficient condition for the existence of pure stabilizer codes with parameters $[[N,k,d]]_q$ is given by Feng and Ma in \cite{feng}.
Assuming $N > k\geq 2$, $d \geq 2$ and $N \equiv k \; (\mathrm{mod} \;2)$, this condition reads
\begin{equation}\label{GV}
\sum_{i=1}^{d-1} (q^2 -1)^{i-1} \binom{N}{i}  < \frac{q^{N-k+2}-1}{q^2-1}.
\end{equation}
In case $N$ odd and $k=1$, the condition is \[q^N +1 > \sum_{i=1}^{d-1} \binom{N}{i} [q (q^2-1)^{i-1} + (-1)^{i+1} (q + 1)^{i-1}]\] and there exists a similar formula for the case $N$ even and $k=0$.


We will use this bound as a measure of goodness of our codes.
We will only consider codes exceeding this bound, i.e.,
cases in which the parameters $q, N, k$ and $d$ satisfy
\begin{equation}\label{GV2}
\sum_{i=1}^{d-1} (q^2 -1)^{i-1} \binom{N}{i}  \geq  \frac{q^{N-k+2}-1}{q^2-1}.
\end{equation}
We will say that an $[[N,k,d]]_q$ quantum code is GV if it fulfills this inequality.

\subsubsection{Curves $A_{n,q,\ell}$}

Let $\ell$ and $n$ be positive integers (not both equal to 1) such that $\gcd(p,\ell)=1$ and let 
$A_{n,q,\ell}$ be the curve defined in Section \ref{Family1}. From Corollary \ref{ccc} and \eqref{eq:qparameters} it is deduced the following result:

\begin{thm}\label{th83}
Let $C(D,m P_{\infty})$ be the code coming from the curve $\tilde{\chi}_{A_{n,q,\ell}}$ over $\mathbb{F}_{q^n}$ as in Section \ref{Family1}. Assume that
$$
2(m+1)\leq (q^{n-1}-1)(\ell-1)+q^{n-1}(\mu\cdot \gcd(\ell (q-1), q^n-1)+1),
$$
where $\mu:=1$ if $p$ divides $\gcd(\ell (q-1), q^n-1)+1$ and $\mu:=0$ otherwise.
Then there exists a quantum code with parameters
$$
[[N,N-2m+2g-2,\ge m-2g+2]]_{q^n},
$$
where 
$N=N(A_{n,q,\ell},q^n)=q^{n-1}\cdot\left(e+1\right)$, $g=(q^{n-1}-1)(\ell-1)/2$ and $e:=\gcd(\ell (q-1), q^n-1)$.

\end{thm}
Notice that, under the hypotheses of the above theorem, $2g-2<m<N$ only if $\mu=1$, and these cases satisfy the hypotheses of \cite[Lemma 2]{Quantum} (they correspond to Type I of Remark \ref{comment}).





First we give an example where $\mu=0$.

\begin{exam}\label{exam:A1}
Consider the curve $A_{2,9,8}$, with equation $y^9 + 2x^8 z + yz^8=0$.
For $m=9<2g-2=54$, the quantum code obtained from Theorem \ref{th83} has 
parameters $[[153,147,3]]_{3^4}$.
The dimension of $C(D,G)$ and the distance of its dual have been computed using MAGMA.
\end{exam}

Next we give an example where $\mu=1$.

\begin{exam}\label{exam:A2}
Consider the curve $A_{2,9,10}$, with equation $y^9z + 2x^{10} + yz^9=0$.
For   $j=0,\ldots,381-194$ we have quantum codes with parameters
$$
[[729,413-2j,123+j]]_{3^4},
$$
which correspond to $194\le m\le 381$. By \eqref{formulabuena} these codes are pure. Moreover they are GV.
\end{exam}

\subsubsection{Curves $B_{q,x^{q^k+1}+x}$}

From Section \ref{Family2} and Eq.~\eqref{eq:qparameters} we get the following:

\begin{thm}
Asume that $\gcd(p,2k)=1$ and $n=2k$. Consider the curve $B_{q,H_k}$, with equation $y^q-y=H_k(x)=x^{q^k+1}+x$, and the code $C(D,m P_{\infty})$ coming from $\tilde{\chi}_{B_{q,H_k}}$ (over $\mathbb{F}_{q^n}$), as defined in Section \ref{Family2}. Assume that 
$$
2(m+1)\leq q^n.
$$
Then there exists a quantum code with parameters
$$
[[q^n+q^k,q^n+q^k-2m+2g-2,\ge m-(q-1)q^k+2]]_{q^n}.
$$
\end{thm}

Now we include some examples of quantum codes coming from the above theorem. 

\begin{exam}\label{ex:B1}
Consider $q=3$, $n=8$, $k=4$. For   $j=0,\ldots,3279-538$ we have quantum codes with parameters 
$$
[[6642,5726-2j,378+j]]_{3^8},
$$
which correspond to $538\le m\le 3279$. By \eqref{formulabuena} these codes are pure. Moreover they are GV.
\end{exam}

\begin{exam}\label{ex:B2}
Consider $q=5$, $n=6$, $k=3$. For  $j=0,\ldots,7811-1955$  we have quantum codes with parameters 
$$
[[15750,12338-2j,1457+j]]_{5^6},
$$
which correspond to $1955\le m\le 7811$. By \eqref{formulabuena} these codes are pure. Moreover they are GV.
\end{exam}

\subsubsection{Curves $C_{q,\ell}$}

From Corollary \ref{ccc3} and Eq.~\eqref{eq:qparameters} we obtain the following result concerning codes coming from curves in the family given in Section \ref{Family31}.

\begin{thm}
Consider the code $C(D,m P_{\infty})$ coming from the curve $\tilde{\chi}_{C_{q,\ell}}$ over $\mathbb{F}_{q^2}$ under the assumptions and hypotheses of Corollary \ref{ccc3}. If 
$$
2(m+1)\leq (q-1)(\ell-1)+q(e+1)-\mu \cdot qe,
$$
where $\mu=0$ if $p$ divides $e+1$ and $\mu=1$ otherwise, then
there exists a quantum code with parameters
$$
[[N,N-2m+2g-2,\ge m-(q-1)(\ell-1)+2]]_{q},
$$
where 
$N=N(C_{q,\ell},q^2)=q(e+1)$ and $e=\gcd(\ell (q-1), q^2-1)$.

\end{thm}


If $\mu=0$ then the value for $m$ is less than $2g-2$ 
so we have to compute the dimension and the minimum distance using MAGMA.

\begin{exam}\label{ex:C1}
Consider  $p=3$, $n=4$, $q=p^2$ and $\ell=5$. 
We choose $m=9$ and so we consider the code $\mathcal{C}=C(D,9P_\infty)$.
Using MAGMA we compute that $N=369$ (thus the curve is maximal) and
$\dim (\mathcal{C})=3$ and $d(\mathcal{C}^{\perp})=3$, so there exists a quantum code with parameters $[[369,363, \ge 3]]_{3^4}$. 
\end{exam}

We present now an example where $\mu=1$; notice that this 
satisfies the hypothesis of \cite[Lemma 2]{Quantum}.

\begin{exam}\label{ex:C2}
Consider $p=3$, $n=6$, $q=p^3$ and $\ell=7$. For 
$j=0,\ldots,2508-313$, there exist quantum codes with parameters
$$
[[ 4941, 4469-2j,159+j]]_{3^6}
$$
which correspond to $313\le m\le 2508$. By \eqref{formulabuena} these codes are pure. Moreover they are GV.
\end{exam}

\subsection{Hermitian Inner Product}



Let $\mathbf{F}=\mathbb{F}_{q^2}$, $q$ being a power of a prime number. Within this framework, the results of Section \ref{sect3} can be applied to obtain quantum codes by using Hermitian inner product instead of Euclidean inner product. 

Notice that ${\bf x} \cdot_h {\bf y}=0$ if and only if ${\bf x}\cdot {\bf y}^q=0$. For any linear code $\mathcal{C}$ we have therefore that
$\mathcal{C}\subseteq \mathcal{C}^{\perp_h}$ if and only if $\mathcal{C}^q\subseteq \mathcal{C}^{\perp}$. For AG codes we have that 
$C(D,G)^q\subseteq C(D,qG)$. Hence if $C(D,qG)\subseteq C(D,G)^{\perp}$ then 
\[
C(D,G)^q\subseteq C(D,qG)\subseteq C(D,G)^{\perp}
\]
and, therefore, $C(D,G)\subseteq C(D,G)^{\perp_h}$.
So we can trivially extend previous results (Theorem \ref{theorem2} and Corollaries \ref{coro}, \ref{coro1}, and \ref{coro2}) using the latter observation.

\begin{thm}\label{theorem91}
Let $\mathbf{F}$ be as before and let $C$, $g$, ${\mathcal A}$, $f'_{\mathcal A}$, $M$, $D$ and $G$ be as in Theorem \ref{theorem2}. Then:
\begin{itemize}
\item[(a)]  If  $(q+1)G\leq(2g-2+\deg(D)-\deg(M))P_{\infty}+M$ then $C(D,G)\subseteq C(D,G)^{\perp_h}$.
\item[(b)] If $G=mP_{\infty}$, with $m\in \mathbb{N}$, and $(q+1)m\leq 2g-2+\deg(D)-\deg(M)$ then $C(D,G)\subseteq C(D,G)^{\perp_h}$.
\item[(c)] If $G=mP_{\infty}$, with $m\in \mathbb{N}$, the curve $C$ has an equation of the type $F(y)=G(x)$, where $F, G$ are polynomials with coefficients in $\mathbf{F}$, and $$(q+1)m\leq 2g-2+\deg(D)-\deg(f'_{\mathcal A})\deg(F)\;\mbox{ or }\;(q+1)m\leq 2g-2+\deg (F),$$ then $C(D,G)\subseteq C(D,G)^{\perp_h}$. 
\end{itemize}
\end{thm}









Notice that the values $m$ satisfying parts (b) and (c) of Theorem \ref{theorem91} are not bigger than $2g-2$. Hence, in practice, we are forced to compute the minimum distance of associated quantum codes with MAGMA. Finally we provide two examples applying Theorem \ref{theorem91}.


\begin{exam}
Applying Theorem \ref{theorem91} to the code of Example \ref{exam:A1} we can produce a quantum code with parameters $[[153,147,3]]_{3^2}$. By \eqref{formulabuena} these codes are pure. Moreover they are GV.
\end{exam}

\begin{exam}\label{ex:D}
Similarly, considering the curve $C_{9,5}$ from Section \ref{Family3} (with equation $y^9 + 2 x^5 z^4 + 2 y z^8=0$), with  $q=9$, $n=4$, $\ell=5$, for $m=9<2g-2=30$ we have the quantum  code $[[369,363,3]]_{3^2}$. By \eqref{formulabuena} these codes are pure. Moreover they are GV.
\end{exam}

\subsection{Comparison with other papers}

As we already mentioned in this paper, we obtain new results for any curve where the divisor $M$ in Theorem \ref{theorem2}  is not equal to zero.  All the examples in  \cite{Quantum} are 
the special case of our results when $M=0$.

Jin and Xing have the following interesting result in  \cite{JinXing}:
If
\[
deg(G)<\frac{N}{2}(1-\frac{1}{N}+log_q(1-\frac{1}{q})-log_q(2)).
\]
and $q$ is even then 
there exists an equivalent code to $C(D,G)$ over $\mathbb{F}_q$ which is Euclidean Self-Orthogonal.

We compare some curves and codes  with this result in the case $G=m P_{\infty}$.
We now present some examples of curves that satisfy the hypotheses of Theorem 3.1 and
where the divisor $M$ is nonzero (checked with Magma).


\begin{exam}\label{exam:26}
Here $q=2$ and $n=6$. The curve is defined by $F(y)=H(x)$ where

$F(y)=y^{16} + y^4  + y$

$H(x)=x^{17} + x^{13} + x^6 + x^4 + x^3 + x+1$.

The Jin-Xing bound implies that there exists a self-orthogonal code (from some curve) for $m\le 66.27$, and our bound 
in Theorem   \ref{theorem2} implies that there exists a self-orthogonal code (from this curve) for $m\le 135.5$.
\end{exam}

There is something of additional interest in the previous  example; our bound is tight.
We confirmed with  Magma  that the AG code is self-orthogonal for $m\le  135$ and NOT self-orthogonal for $m=136$.
This shows that, in some sense, our bound cannot be improved.

\begin{exam}\label{exam:27}
$q=2$, $n=7$

$F(y)= y^8 + y^4 + y $

$H(x)=x^{29} + x^{17} + x^{15} + x^{13} + x^{12} + x^{11} + x^9 + x^4 + x+1$, 

See Table 1 for the range of values for $m$.
\end{exam}

\begin{exam}\label{exam:28}
$q=2$, $n=8$

$F(y)=y^{64}+y^{16}+y^4+y  $

$H(x)=x^{67} + x^{63} + x^{61} + x^{58} + x^{56} + x^{54} + x^{53} + x^{52} + x^{50} + x^{49} + x^{48}
        + x^{46} + x^{45} + x^{44} + x^{39} + x^{37} + x^{36} + x^{35} + x^{34} + x^{33} + x^{31} +
        x^{29} + x^{23} + x^{20} + x^{16} + x^{15} + x^{14} + x^{10} + x^8 + x^4 + x^3+1 $, 

See Table 1 for the range of values for $m$.
\end{exam}

\begin{table}[htbp]
\begin{center} 
\begin{tabular}{|l||l||l||l||l||l||l||l|l|}
\hline
  &    \cite{JinXing} & This paper \\
\hline \hline
Example \ref{exam:max} &\text{nothing for $q$ odd}&$17 \le m \le 129$ \\ \hline
Example \ref{exam:26} & $m\le 66 $  & $m\le 135$  \\ \hline
Example \ref{exam:27} & $m\le 78 $  & $m\le 101$  \\ \hline
Example \ref{exam:28} & $m\le 1902 $  & $m\le 2398$  \\ \hline
Example \ref{ex:B1} & $\text{nothing for $q$ odd} $  & $m\le 3279$  \\ \hline
Example \ref{ex:B2} & $\text{nothing for $q$ odd} $  & $m\le 7811$  \\ \hline
Example \ref{ex:C1} & $\text{nothing for $q$ odd} $  & $m\le 2508$  \\ \hline
\end{tabular}
\end{center}
Table 1
\end{table}

Any example with $q$ odd will improve on  \cite{JinXing} because the bound in that
paper for Euclidean codes is not valid when $q$ is odd.
From our infinite families we list three examples in the table.
They do not appear in  \cite{Quantum}  because the divisor $M$ is nonzero, 
as we proved in the earlier sections.

 {\bf Acknowledgements} We thank J.I.~Farr\'an and C.~Munuera  for helpful conversations.

\end{document}